\documentclass[a4paper,10pt]{article}
\usepackage[T1]{fontenc}
\usepackage[utf8]{inputenc}
\usepackage{latexsym,amssymb}
\usepackage{graphicx}
\usepackage{hyperref}
\usepackage{amsmath}
\usepackage[english]{babel}
\usepackage{tikz-cd}
\usepackage{scalerel,stackengine}
\usepackage{bm}
\usepackage{url}
\usepackage{amsmath}
\usepackage{amssymb}
\usepackage{amsthm}
\usepackage{mathtools}
\usepackage{amsmath,calligra,mathrsfs}
\usepackage{yhmath}
\usepackage{bashful}
\usepackage{hyperref}
\newtheorem{theorem}{Theorem}
\newtheorem{lemma}[theorem]{Lemma}

\theoremstyle{definition}

\newtheorem{remark}[theorem]{Remark}
\newtheorem{conjecture}[theorem]{Conjecture}

\makeatletter
\author{Diana Mocanu}
\newcommand*{\inlineequation}[2][]{%
  \begingroup
    \refstepcounter{equation}%
    \ifx\\#1\\%
    \else
      \label{#1}%
    \fi
    \relpenalty=10000 %
    \binoppenalty=10000 %
    \ensuremath{%
      #2%
    }%
     ~\@eqnnum
  \endgroup
}
\makeatother

\title{ \bf{Asymptotic Fermat for signatures $(p,p,2)$ and $(p,p,3)$ over totally real fields} 
\noindent}
\newcommand{\rep}{\overline{\rho}}
\newcommand{\fq}{\mathfrak{q}}
\newcommand{\fP}{\mathfrak{P}}
\newcommand{\fp}{\mathfrak{p}}

\newcommand{\cN}{\mathcal{N}}
\newcommand{\cO}{\mathcal{O}}
\newcommand{\N}{\text{Norm}}
\author{Diana Mocanu}
\date{}

\begin{document}

\maketitle
\begin{abstract}
   Let $K$ be a totally real number field and consider a Fermat-type equation $Aa^p+Bb^q=Cc^r$ over $K$. We call the triple of exponents $(p,q,r)$ the \textit{signature} of the equation. We prove various results concerning the solutions to the Fermat equation with signature $(p,p,2)$ and $(p,p,3)$ using a method involving modularity, level lowering and image of inertia comparison. These generalize and extend the recent work of Işik, Kara and Özman. For example, consider $K$ a totally real field of degree $n$ with $2 \nmid h_K^+$ and $2$ inert. Moreover, suppose there is a prime $q\geq 5$ which totally ramifies in $K$ and satisfies $\gcd(n,q-1)=1$, then we know that the equation $a^p+b^p=c^2$ has no primitive, non-trivial solutions $(a,b,c) \in \cO_K^3$ with $2 | b$ for $p$ sufficiently large.
\end{abstract}    
\section{Introduction}
\subsection{Historical background}
The study of Diophantine equations is of great interest 
in Mathematics and goes back to antiquity. The most famous example of a Diophantine equation appears in \textit{Fermat’s Last Theorem}. This is the statement, asserted by Fermat in 1637 without proof, that the Diophantine equation $a^n+b^n=c^n$ has no solutions in whole numbers when $n$ is at least $3$, other than the trivial solutions which arise when $abc=0$. Andrew Wiles famously proved the Fermat's Last Theorem in 1995 in his paper "Modular elliptic curves and Fermat’s Last Theorem" \cite{45}. The proof is by contradiction employing techniques from algebraic geometry and number theory to prove a special case of the modularity theorem for elliptic curves, which together with Ribet’s level lowering theorem gives the long-waited result. Since then, number theorists extensively studied Diophantine equations using Wiles' modularity approach. Siksek gives a comprehensive survey about this method over the field of rationals in \cite{SS}.\par 
Even before Wiles announced his proof, various generalizations of Fermat’s Last Theorem had already been
considered, which are of the shape
\begin{equation} \label{intro}
    Aa^p+Bb^q=Cc^r
\end{equation}
for fixed integers $A,B$ and $C$. We call $(p,q,r)$ \textit{the signature} of the equation (\ref{intro}). A \textit{primitive} solution $(a,b,c)$ is a solution where $a,b$ and $c$ are pairwise coprime and a \textit{non-trivial} solution $(a,b,c)$ is a solution where $abc \neq 0$. \par  
In \cite{T},  Işik, Kara and Özman list all known cases where equation (\ref{intro}) has been solved over the rational integers in two tables (p.4). Table 1 contains all unconditional results for infinitely many primes. In Table 2, they give all conditional results. We highlight here one relevant family of solutions, namely $(n,n,k)$ where $k\in \{2,3 \}$. Darmon and Merel \cite{DM} and Poonen \cite{P} proved the following theorem:
\begin{theorem}[Darmon and Merel]\label{DMP}
\begin{enumerate}
    \item The equation $a^n+b^n=c^2$ has no non-trivial primitive integer solutions for $n\geq4$.
    \item The equation $a^n+b^n=c^3$ has no non-trivial primitive integer solutions for $n\geq3$.
\end{enumerate}
\end{theorem}
Note that the above equations, typically have infinitely many non-primitive solutions. For example, if $n$ is odd, and $a$ and $b$ are any two
integers with $a^n + b^n = c$, then 
$$ (ac)^n+(bc)^n=(c^{\frac{n+1}{2}})^2
$$
giving a rather uninteresting supply of solutions. Thus, we would only study the primitive solutions of the above equations. 
\par A naive sketch of the proof of Theorem \ref{DMP} is as follows. First note that it is enough to prove the assumption for $n=p$ an odd prime. Suppose $a,b,c\in \mathbb{Z}$ is a non-trivial, primitive solution to (i) or (ii). In each of the cases, we can associate a so-called Frey elliptic curve $E_{a,b,c}/\mathbb{Q}$ and let $\rep_{E,p}$ be its$\mod p$ Galois representation, where $E=E_{a,b,c}$. Then $\rep_{E,p}$ is irreducible by Mazur \cite{32} and modular by Wiles and Taylor \cite{45} and \cite{43}. Applying Ribet's level lowering theorem \cite{36} one gets that that $\rep_{E,p}$ arises from a weight $2$ newform
of level $32$ for (i) and level $27$ for (ii). These are closely related to the modular curves $X_0(32)$ and $X_0(27)$ which turn out to be elliptic curves with complex multiplication. Darmon and Merel prove in \cite{DM}, by using the theory of complex multiplication that this implies $j_E\in \mathbb{Z}[\frac{1}{p}]$ for  $p>7$, which gives a contradiction. The cases when $p\leq 7$ are treated in a more elementary way by Poonen \cite{P}.
\par 
Recently, important progress has been done towards generalisation of the modularity approach over larger number fields. 
In \cite{SN} Freitas and Siksek proved \textit{the asymptotic
Fermat’s Last Theorem (AFLT)} for certain totally real fields $K.$ That is, they showed that there is a constant $B_K$
such that for any prime $p>B_K$, the only solutions to the Fermat equation 
$a^p + b^p + c^p = 0$ where $a,b,c \in \cO_K$ are the trivial ones i.e. the ones satisfying $abc=0$.
Then, Deconinck \cite{De} extended the results of Freitas and Siksek \cite{SN} to
the generalized Fermat equation of the form $Aa^p+Bb^p+Cc^p=0$ where $A, B, C$ are odd integers belonging to a totally real field. Later in \cite{Se} Şengün and Siksek proved the asymptotic FLT for any number field $K$ by
assuming modularity. This result has been generalized by Kara and Özman in \cite{Ka} to the case of the generalized Fermat equation. Also, recently in \cite{Tu1} and \cite{Tu2} Țurcaș studied Fermat equation over imaginary quadratic field $\mathbb{Q}(\sqrt{-d})$ with class number one.\par 
We now present a result by Işik, Kara and Özman, proved in \cite{T} which serves as the starting point of this paper.
It gives a computable criteria of testing if the \textit{asymptotic Fermat Last Theorem} holds for certain type of solutions of the equations with signatures $(p,p,2)$.
To state it, we need the following notation:
$$S_K:=\{\mathfrak{P}\: : \mathfrak{P} \text{ is a prime of } K \text{ above } 2\},\:\:\: T_K:=\{\mathfrak{P} \in S_K : f(\mathfrak{P}/2)=1\},$$
$$W_K:=\{(a,b,c) \in \mathcal{O}_K^ 3 : a^p+b^p=c^2 \text{ with } \mathfrak{P}|b \text{ for every } \mathfrak{P} \in T_K\}$$
where $f(\mathfrak{P}/2)$ denotes the residual degree of $\mathfrak{P}$.
\begin{theorem}[Işik, Kara and Özman]\label{turkish}
Let $K$ be a totally real number field with narrow class number $h_K^+=1$.
For each $a \in K(S_K,2)$, let $L=K(\sqrt{a})$. \\(A):Suppose that for every solution $(\lambda,\mu)$ to the $S_K$-unit equation
$$ \lambda+\mu=1, \: \lambda, \mu \in \mathcal{O}_{S_K}^*
$$
there is some $\mathfrak{P} \in T_K$ that satisfies max\{$|v_{\mathfrak{P}}(\lambda)|,|v_{\mathfrak{P}}(\mu)|$\}$\leq 4v_{\mathfrak{P}}(2)$.\\
(B):Suppose also that for each $L$, for every solution $(\lambda,\mu)$ of the $S_L$-unit equation $ \lambda+\mu=1,\: \lambda, \mu \in \mathcal{O}_{S_L}^*$, there is some $\mathfrak{P}'\in T_L$ that satisfies max\{$|v_{\mathfrak{P'}}(\lambda)|,|v_{\mathfrak{P'}}(\mu)|$\}$\leq 4v_{\mathfrak{P'}}(2)$.\\
Then, there is a constant $B_K$  (depending only on $K$) such that for each $p>B_K$, the equation $a^p+b^p=c^2$ has no primitive, non-trivial solutions with $(a,b,c) \in W_K$ (i.e. the asymptotic Fermat holds for $W_K$).
\end{theorem}
\subsection{Our results}
We start by using the methods pioneered by Freitas and Siksek in \cite{SN} involving modularity, level lowering and image of inertia comparison to generalize Işik, Kara and Özman's Theorem \ref{turkish}. More precisely, we relax the assumption on the class group from $h_K^+=1$ to $Cl_{S_K}(K)[2]=\{1\}.$ We use $Cl_S(K)$ to mean $Cl(K)/\langle{[\mathfrak{P}]}\rangle_{\mathfrak{P}\in S}$ for $S$ a finite set of primes of $K$ and consequently, $Cl_S(K)[n]$ denotes its $n$-torsion points. Note that when all $\fP \in S$ are principal,  $Cl_S(K)$ is the usual $Cl(K)$, and hence we will drop the $S$ in the notation. Moreover, in this case, $Cl(K)[p]=\{1\}$ is equivalent to $p \nmid h_K$, for $p$ prime.

Our main theorem regarding the Asymptotic Fermat Last Theorem for signature $(p,p,2)$ reads as follows:
\begin{theorem}[Main Theorem for $(p,p,2)$] \label{main1}
Let $K$ be a totally real number field with $Cl_{S_K}(K)[2]=\{1\}$ where $S_K:=\{\mathfrak{P}\: : \mathfrak{P} \text{ is a prime of } K \text{ above } 2\}$.
Suppose that there exists some distinguished prime $\tilde{\mathfrak{P}} \in S_K$, such that every solution $(\alpha,\beta,\gamma) \in \mathcal{O}_{S_K}^* \times \mathcal{O}_{S_K}^* \times \mathcal{O}_{S_K}$ to the equation
$$ \alpha+\beta=\gamma^2
$$
satisfies $|v_{\tilde{\mathfrak{P}}}(\frac{\alpha}{\beta})|\leq 6v_{\tilde{\mathfrak{P}}}(2)$. Then, there is a constant $B_K$ (depending only on $K$) such that for each rational prime $p>B_K$, the equation $a^p+b^p=c^2$ has no primitive, non-trivial solutions $(a,b,c) \in \mathcal{O}_K^3$ with $\tilde{\mathfrak{P}}|b$.
\end{theorem}
\begin{remark}
By Theorem \ref{equation} the equation \[\alpha+\beta=\gamma^2, \qquad(\alpha, \beta, \gamma) \in \mathcal{O}_{S_K}^* \times \mathcal{O}_{S_K}^* \times \mathcal{O}_{S_K}\] has finitely many solutions up to scaling by a square in $\mathcal{O}_{S_K}^*$, and these are effectively computable.
Hence the criteria in Theorem \ref{main1} is testable in finite time.
\end{remark}
Imposing local constraints, we get that for a totally real number field, in which $2$ is inert, the following holds:
\begin{theorem}\label{2inert}
Let K be a totally real number field with $2 \nmid h_K^+$ in which $2$ is inert. Let $\mathfrak{P}$ be the only prime above $2$, and hence $S_K = \{ \mathfrak{P} \}$.
Suppose that every solution $(\alpha, \gamma) \in \mathcal{O}^*_{S_K}\times \cO_{S_K}$ with $v_{\mathfrak{P}}(\alpha)\geq 0$ to the equation
\begin{equation}\label{2ineq}
\alpha + 1 = \gamma^2
\end{equation}
satisfies $v_{\fP}(\alpha)\leq 6$.
Then, there is a constant $B_K$ (depending only on $K$) such that for each rational prime $p>B_K$, the equation $a^p+b^p=c^2$ has no primitive, non-trivial solutions $(a,b,c)\in \cO_K^3$ with $2 | b$.
\end{theorem}

More concretely, for quadratic totally real number fields $K$, Theorem \ref{2inert} becomes:

\begin{theorem}\label{2quad}
Let $d>5$ be a rational prime satisfying $d\equiv 5 \mod{8}$. Write $K=\mathbb{Q}(\sqrt{d})$.
Then, there is a constant $B_K$  (depending only on $K$) such that for each rational prime $p>B_K$, the equation $a^p+b^p=c^2$ has no primitive, non-trivial solutions $(a,b,c) \in \cO_K^3$ with $2 | b$.
\end{theorem}
More generally, by employing additional local information, the following holds.
\begin{theorem}\label{2local}
Let $K$ be a totally real field of degree $n$, and let $q\geq 5$ be a rational prime. Suppose 
\begin{enumerate}
    \item $2 \nmid h_K^+$,
    \item $\gcd(n,q-1)=1$,
    \item $2$ is inert in $K$,
    \item $q$ totally ramifies in $K$.
\end{enumerate}
Then, there is a constant $B_K$  (depending only on $K$) such that for each rational prime $p>B_K$, the equation $a^p+b^p=c^2$ has no primitive, non-trivial solutions $(a,b,c) \in \cO_K^3$ with $2 | b$.
\end{theorem}
\begin{remark}
A few examples of totally real fields $K$ satisfying the conditions above are the degree $3$ extensions of narrow class number $1$, which have the following defining polynomials and totally ramified prime $q$:
\begin{itemize}
    \item $p_1(x)=x^3-51x-85 \:\:(q=17)$,
    \item $p_2(x)=x^3-x^2-40x+13\:\:(q=11)$,
    \item $p_3(x)=x^3-x^2-38x-75\:\:(q=23)$,
    \item $p_4(x)=x^3-17x-17\:\:(q=17)$.
\end{itemize}

\end{remark}
We use the same methods to study the asymptotic behaviour of the analogue $(p,p,3)$ equation and we get the following:
\begin{theorem}[Main Theorem for $(p,p,3)$] \label{main2}
Let $K$ be a totally real number field with $Cl_{S_K}(K)[3]=\{1\}$ where $S_K:=\{\mathfrak{P}\: : \mathfrak{P} \text{ is a prime of } K \text{ above } 3\}$.
Suppose that there exists some distinguished prime $\tilde{\mathfrak{P}} \in S_K$ such that every solution $(\alpha,\beta,\gamma) \in \mathcal{O}_{S_K}^* \times \mathcal{O}_{S_K}^* \times \mathcal{O}_{S_K}$ to the $S_K$ equation
$$ \alpha+\beta=\gamma^3
$$
satisfies $|v_{\tilde{\mathfrak{P}}}(\frac{\alpha}{\beta})|\leq 3v_{\tilde{\mathfrak{P}}}(3)$.
Then, there is a constant $B_K$  (depending only on $K$) such that for each rational prime $p>B_K$, the equation $a^p+b^p=c^3$ has no primitive, non-trivial solutions $(a,b,c) \in \mathcal{O}_K^3$ with $\tilde{\mathfrak{P}}|b$.
\end{theorem}

\begin{remark}
By Theorem \ref{equation} the equation \[\alpha+\beta=\gamma^3,\qquad (\alpha, \beta, \gamma) \in \mathcal{O}_{S_K}^* \times \mathcal{O}_{S_K}^* \times \mathcal{O}_{S_K}\] has finitely many solutions up to scaling by a cube in $\mathcal{O}_{S_K}^*$, and these are effectively computable.
Hence the criteria in Theorem \ref{main2} is testable in finite time.
\end{remark}
Similarly to the $(p,p,2)$ case, the following hold when employing local information. We will consider various field extensions involving the primitive cube root of unity $\omega := \cos (\frac{2\pi}{3})+i\sin(\frac{2\pi}{3})$.
\begin{theorem}\label{3inert}
Let $K$ be a totally real number field such that $3 \nmid h_{K(\omega)}$, $3 \nmid h_K$ and in which $3$ is inert. Let $\mathfrak{P}$ be the only prime above $3$, and hence $S_K = \{ \mathfrak{P} \}$.
Suppose that every solution $(\alpha, \gamma) \in \mathcal{O}^*_{S_K}\times \cO_{S_K}$ with $v_{\mathfrak{P}}(\alpha)\geq 0$ to the equation
\begin{equation}\label{3ineq}
     \alpha + 1 = \gamma^3
\end{equation}
satisfies $v_{\fP}(\alpha)\leq 3$.
Then, there is a constant $B_K$  (depending only on $K$) such that for each rational prime $p>B_K$, the equation $a^p+b^p=c^3$ has no primitive, non-trivial solutions $(a,b,c) \in \cO_K^3$ with $3 | b$.
\end{theorem}

\begin{theorem}\label{3quad}
Let $d$ a positive, square-free satisfying $d\equiv 2 \mod{3}$. Write $K=\mathbb{Q}(\sqrt{d})$ and suppose $3 \nmid h_{K(\omega)}$, $3\nmid h_K$.
Then, there is a constant $B_K$  (depending only on $K$) such that for each rational prime $p>B_K$, the equation $a^p+b^p=c^3$ has no primitive, non-trivial solutions $(a,b,c) \in \cO_K^3$ with $3 | b$.
\end{theorem}

\begin{theorem}\label{3local}
Let $K$ be a totally real field of degree $n$, and let $q \geq 5$ be a rational prime. Suppose 
\begin{enumerate}
    \item $3 \nmid h_{K(\omega)}$ and $3 \nmid h_K$,
    \item $\gcd(n,q^2-1)=1$,
    \item $3$ is inert in $K$,
    \item $q$ totally ramifies in $K$.
\end{enumerate}
Then, there is a constant $B_K$  (depending only on $K$) such that for each rational prime $p>B_K$, the equation $a^p+b^p=c^3$ has no primitive, non-trivial solutions $(a,b,c) \in \cO_K^3$ with $3 | b$.
\end{theorem}
\begin{remark}
A few examples of totally real fields $K$ satisfying the conditions above are the degree $5$ extensions with totally ramified prime $q=5$, which have the following defining polynomials and corresponding $h_K, h_{K(\omega)}$:

\begin{itemize}
    \item $p_1(x)=x^5 - 25x^3 - 10x^2 + 50x - 20\:\: (h_K=1, h_{K(\omega)}=29)$,
    \item $p_2(x)=x^5 - 30x^3 - 20x^2 + 160x + 128 \:\:(h_K=1, h_{K(\omega)}=29)$,
    \item $p_3(x)=x^5 - 15x^3 - 10x^2 + 10x + 4 \:\:(h_K=1, h_{K(\omega)}=31)$,
    \item $p_4(x)=x^5 - 20x^3 - 15x^2 + 10x + 4 \:\:(h_K=1, h_{K(\omega)}=361)$.
\end{itemize}
\end{remark}
\subsection{Recent progress}
More recently, Işik, Kara and Özman proved in \cite{IKO} a similar asymptotic result for signature $(p,p,3)$ over general number fields $K$ with narrow class number one satisfying some technical conditions. In the appendix, they show how this result can be adapted to signature $(p,p,2)$. These results use standard modularity conjectures and the study of Bianchi newforms.
\subsection{Notational conventions}
We will follow the notational conventions in \cite{SN}. Throughout $p$ denotes a rational prime, and $K$ a totally real number field, with ring of integers $\mathcal{O}_K$. For a non-zero ideal $I$ of $\mathcal{O}_K$ ,
we denote by $[I]$ the class of $I$ in the class group Cl($K$).\\
Let $G_K=\text{Gal}(\bar{K}/K)$. For an elliptic curve $E/K$, we write 
\begin{equation*}
    \overline{\rho}_{E,p}: G_K \to \text{Aut}(E[p])\simeq \text{GL}_2(\mathbb{F}_p)
\end{equation*}
for the representation of $G_K$ on the $p$-torsion of $E$. For a Hilbert eigenform $\mathfrak{f}$ over $K$, we let $\mathbb{Q}_{\mathfrak{f}}$ denote the field generated by its eigenvalues. In this situation $\varpi$ will denote a prime of $\mathbb{Q}_{\mathfrak{f}}$ above $p$; of course if $\mathbb{Q}_{\mathfrak{f}}=\mathbb{Q}$ we write $p$ instead of $\varpi$. All other
primes we consider are primes of $K$. We reserve the symbol $\mathfrak{P}$ for primes belonging
to $S$. An arbitrary prime of $K$ is denoted by $\mathfrak{q}$,
and $G_{\mathfrak{q}}$ and $I_{\mathfrak{q}}$ are the decomposition and inertia subgroups of $G_K$ at $\mathfrak{q}$.

\textbf{Acknowledgments.} I am sincerely grateful to my supervisor Samir Siksek for his continuous support, useful comments and reviewing this paper.

\section{Preliminaries}

\subsection{Elliptic Curves}
We begin by collecting some useful results about elliptic curves, as they play a key role in the modular approach of solving Diophantine equations. 
\begin{lemma}\label{torsion}
Let $K$ be a field of char$(K) \neq 2,3$ and $E/K$ an elliptic curve. The following holds:
\begin{enumerate}
    \item If $E$ has a $K$-rational point of order $2$, then E has a model of the form
    \begin{equation}\label{torsion2}
         E: Y^2=X^3+aX^2+bX.
    \end{equation}
    Moreover, there is a bijection between
    $$ \{ E/K \text{ with a } K\text{-torsion of order } 2 \text{ up to } \Bar{K}-\text{isomorphism} \} 
    \to \mathbb{P}^1(K) - \{ 4, \infty \} $$ via the map $E \to \lambda:=\frac{a^2}{b}.$
    \item If $E$ has a $K$-rational point of order $3$, then E has a model of the form  
    \begin{equation}\label{torsion3}
        E: Y^2 +cXY +dY=X^3.
    \end{equation}
    Moreover, there is a bijection between
    $$ \{ E/K \text{ with a } K\text{-torsion of order } 3 \text{ up to } \Bar{K}-\text{isomorphism} \} 
    \to \mathbb{P}^1(K) - \{ 27, \infty \} $$ via the map $E \to\lambda:= \frac{c^3}{d}$.
\end{enumerate}
\end{lemma}
\begin{proof}
\begin{enumerate}
    \item The first part is a well-known result.
   For the second part, we are given an elliptic curve $E/K$ with a $K$-torsion point of order $2$. After writing it as in (\ref{torsion2}), we make the assignment $E \mapsto \lambda := \frac{a^2}{b}.$ As $\Delta_{E}=2^{4}b^2(a^2-4b)$, non-singularity of $E$ gives $\lambda \in \mathbb{P}^1(K)-\{4,\infty \}$, which proves our map is well-defined.
   Moreover, any $\lambda \in \mathbb{P}^1(K)-\{4,\infty \}$ can be written as a ratio of the form $\frac{a^2}{b}$ with $b \neq 0$ and $a^2 \neq 4b$, and hence comes from an elliptic
curve with a $K$-rational $2$-torsion. Thus, our map is surjective.

Injectivity follows from writing $$j_{E}=2^8 \frac{(a^2-3b)^3}{b^2(a^2-4b)}=2^8\frac{(\lambda-3)^3}{\lambda-4}$$ and noting that $\lambda = \lambda'$ for given $E\to \lambda$, $E'\to \lambda'$ implies $j_{E}=j_{E'}$, which gives $E \simeq E'$.

    \item If $E$ is in Weierstrass form we can translate the $K$-torsion point to $(0,0)$. This will give a model of the form 
    $$E: Y^2+a_1XY+a_3Y=X^3+a_2X^2+a_4X$$
     We now impose the condition that $(0,0)$ has order $3$. First, we compute $-(0,0)=(0,-a_3)$ and note that we require $(0,0)\neq -(0,0)=(0,-a_3)$, so $a_3 \neq 0$. Now, by performing the change of variables
    \begin{equation}
        \begin{cases}
        Y \to (Y+\frac{a_4}{a_3}X)\\
        X \to X
        \end{cases}
    \end{equation}
    we get a model of the form
    $$E: Y^2+cXY+dY=X^3+eX^2 \text{ with } d=a_3 \neq 0.$$ Finally, we make use of the order $3$,
    \begin{equation}
        \begin{cases}
        (0,0)+(0,0)=-(0,0)=(0,-d)\\
        (0,0)+(0,0)=(-e,-d)
        \end{cases}
    \end{equation}
    Hence, we need $e=0$, and we get the desired form: $E: Y^2+cXY+dY=X^3$. \par 
    For the second part, we are given an elliptic curve $E/K$ with a $K$-torsion point of order $3$. After writing it as in (\ref{torsion3}), we make the assignment $E \mapsto \lambda := \frac{c^3}{d}.$ As $\Delta_{E}=d^3(c^3-27d)$, non-singularity of $E$ gives $\lambda \in \mathbb{P}^1(K)-\{27,\infty \}$, which proves our map is well-defined.
   Moreover, any $\lambda \in \mathbb{P}^1(K)-\{27,\infty \}$ can be written as a ratio of the form $\frac{c^3}{d}$ with $d \neq 0$ and $c^3 \neq 27d$, and hence comes from an elliptic
curve with a $K$-rational $3$-torsion. Thus, our map is surjective.\par 
Injectivity follows from writing $$j_{E}=\frac{c^3(c^3-24d)^3}{d^3(c^3-27d)}=\frac{\lambda(\lambda-24)^3}{\lambda-27}$$ and noting that $\lambda = \lambda'$ for given $E\to \lambda$, $E'\to \lambda'$ implies $j_{E}=j_{E'}$, which gives $E \simeq E'$.
  
\end{enumerate}
\end{proof}
\begin{lemma}\label{elliptic}
Let $K$ be a number field and $S$ a set of finite primes of $K$. Then:
\begin{enumerate}
    \item If $S$ contains the primes above $2$ we get the following bijection\\ 
    $\left\{
\begin{tabular}{@{}l@{}}
   $E/K$\text{ with a }$K$\text{-torsion of order }$2$ with potentially \\ \text{good reduction outside }$S$\text{ up to }$\Bar{K}$-\text{isomorphism}
\end{tabular} 
\right\} \longmapsto \cO^*_S$ via the map $E \to \mu:=\lambda-4 \in \cO^*_S$, where $\lambda$ is as in Lemma \ref{torsion} (i). \\
\item If $S$ contains the primes above $3$ we get the following bijection\\ 
    $\left\{
\begin{tabular}{@{}l@{}}
   $E/K$\text{ with a }$K$\text{-torsion of order }$3$ with potentially \\ \text{good reduction outside }$S$\text{ up to }$\Bar{K}$-\text{isomorphism}
\end{tabular} 
\right\} \longmapsto \cO^*_S$ via the map $E \to \mu:=\lambda-27 \in \cO^*_S$, where $\lambda$ is as in Lemma \ref{torsion} (ii). \\
\end{enumerate}
\end{lemma}
\begin{proof}
\begin{enumerate}
    \item Let $E$ be an elliptic curve with a $K$-torsion point of order $2$ with potentially 
    good reduction outside $S$. By Lemma \ref{torsion} (i) $E$ has a model
    $$E: Y^2=X^3+aX^2+bX$$
    with $\lambda:= \frac{a^2}{b}$ and $\mu:=\lambda-4=\frac{a^2-4b}{b}$.
    Thus
    \begin{equation}\label{jinvariant}
    j_{E}=2^8\frac{(\lambda-3)^3}{\lambda-4}=2^8\frac{(\mu+1)^3}{\mu}.
\end{equation}
Potentially good reduction outside $S$ implies that $v_{\fq}(j_E)\geq0$ for all $\fq \notin S$, in other words $j_{E} \in  \mathcal{O}_{S}$. Consequently both $\lambda$ and $\mu$ satisfy monic equations with coefficients in $\mathcal{O}_{S}$. Thus, we can conclude that $\lambda, \mu \in \mathcal{O}_{S}$. Moreover, by writing $j_{E}$ in terms of $\mu^{-1}$ and using the same reasoning, we deduce that also $\mu^{-1} \in \mathcal{O}_{S}$ and hence $\mu \in \mathcal{O}_{S}^*$ and so the assignment $E \longmapsto \mu$ is well-defined. 

Note that every $\mu \in \cO^*_S$ can be written in the form $\mu = \frac{a^2}{b}-4$ for some $a, b \in K$, thus coming from an elliptic curve with $2$-torsion. Moreover, $\mu \in \cO^*_S$ implies $j_E \in \cO_S$, thus this represents a curve with potentially good reduction outside $S$, proving surjectivity.

Injectivity follows by noting that $\mu=\mu'$ implies $j_E=j_{E'}$ which gives $E \simeq E'$.  
    \item Let $E$ be an elliptic curve with a $K$-torsion point of order $3$ with potentially 
    good reduction outside $S$. By Lemma \ref{torsion} (ii) $E$ has a model
    $$E: Y^2 +cXY +dY=X^3$$
    with $\lambda:=\frac{c^3}{d} $ and $\mu=\lambda-27=\frac{c^3-27d}{d}$
    Thus, 
    \begin{equation}\label{jinvariant2}
    j_{E}=\frac{\lambda(\lambda-24)^3}{\lambda-27}=\frac{(\mu+27)(\mu+3)^3}{\mu}.
\end{equation}

 Same arguments as in the proof of (i) give $j_E, \lambda \in \cO_S$ and $\mu \in \cO^*_S$, giving $E \longmapsto \mu$ is well-defined.
 
  Surjectivity and injectivity follow exactly as in (i).
\end{enumerate}
\end{proof}
We say that a fractional ideal is an \textit{$S$-ideal} if its decomposition into primes contains only primes in $S$.
\begin{lemma}\label{mainelliptic}
Let $K$ be a number field and $S$ a set of finite primes of $K$. Let $E/K$ be an elliptic curve with good reduction outside $S$.
\begin{enumerate}
    \item Suppose $S$ contains the primes above $2$ and $E$ has a $K$-torsion point of order $2$. Let $(\lambda, \mu) \in \mathcal{O}_S \times \mathcal{O}_S^*$ correspond to $E$ as in Lemma \ref{elliptic} (i) and therefore satisfy $\lambda-\mu=4$.
   Then $(\lambda)\cO_K=I^2J$ where $I,J$ are fractional ideals with $J$ being an $S$-ideal.
    \item Suppose $S$ contains the primes above $3$ and $E$ has a $K$-torsion point of order $3$. Let $(\lambda, \mu) \in \mathcal{O}_S \times \mathcal{O}_S^*$ correspond to $E$ as in Lemma \ref{elliptic} (ii) and therefore satisfy $\lambda-\mu=27$.
   Then $(\lambda)\cO_K=I^3J$ where $I,J$ are fractional ideals with $J$ being an $S$-ideal.
\end{enumerate}
\end{lemma}
\begin{proof}
\begin{enumerate}
    \item
    By Lemma \ref{torsion} (i) $E$ has a model
    $$E: Y^2=X^3+aX^2+bX$$ with $\Delta_{E}=2^{4}b^2(a^2-4b)$ and $c_4=2^4(a^2-3b)$.
    Good reduction outside $S$ implies that for a $\mathfrak{q}\notin S$ we have that $v_{\fq}(\Delta_{\text{min}})=0$ (where $\Delta_{\text{min}}$ is the minimal discriminant of $E$ viewed over the local field $K_{\fq}$). Standard results about the minimal discriminant of an elliptic curve (e.g. \cite[Ch. VII.1.]{Silv}) give $\fq^{12k} || \Delta_E$ and $\fq^{4k} | c_4$. As $\mathfrak{q}$ is an odd prime, this yields to the following two relations $$\fq^{12k} || b^2(a^2-4b), \quad \fq^{4k} | (a^2-3b).$$
    
    Now, we claim that $\fq^{4k} | b$. Suppose not, by the first relation it follows that $\fq^{4k}|(a^2-4b)$ and combining this with the second relation we get that $\fq^{4k} | (a^2-3b) - (a^2-4b)= b$, a contradiction. Hence, $v_{\fq}(b):=t\geq 4k$. Observe that the second relation implies $v_{\fq}(a^2-3b):=s\geq 4k$. By the first relation 
    $$ 12k=v_{\fq}(b^2(a^2-4b)) = 2t+ v_{\fq}(a^2-3b-b) \geq 2t +\min(s,t) \geq 2t + 4k.
    $$
    This implies $t\leq 4k$, which gives $t=4k$ (as we have already shown $t\geq 4k$). Moreover the second relation implies $\fq^{4k}| a^2$.
   Therefore, we can conclude $\mathfrak{q}^{2k}|a$ and  $\mathfrak{q}^{4k}||b$.
Hence, $$(a)\mathcal{O}_K=\displaystyle\prod_{\mathfrak{q}\notin S_K}\mathfrak{q}^{2k_{\mathfrak{q}}+l_{\mathfrak{q}}} \prod_{\mathfrak{P}\in S_K}\mathfrak{P}^{a_\mathfrak{P}}, \:  (b)\mathcal{O}_K=\displaystyle\prod_{\mathfrak{q}\notin S_K}\mathfrak{q}^{4k_{\mathfrak{q}}} \prod_{\mathfrak{P}\in S_K}\mathfrak{P}^{b_\mathfrak{P}}$$ for some positive integers $a_{\mathfrak{P}}, b_{\mathfrak{P}},k_{\mathfrak{q}}, l_{\mathfrak{q}}$. Thus, as $\lambda=\frac{a^2}{b}$, we get
$$(\lambda)\mathcal{O}_K=I^2J \text{, where } I:=\displaystyle\prod_{\mathfrak{q}\notin S_K}\mathfrak{q}^{l_{\mathfrak{q}}}, \: J:=\prod_{\mathfrak{P}\in S_K}\mathfrak{P}^{2a_\mathfrak{P}-b_\mathfrak{P}}$$
which makes $J$ an $S$-ideal. 
\item  By Lemma \ref{torsion} (ii) $E$ has a model
    $$E: Y^2 +cXY +dY=X^3$$ with $\Delta_{E}=d^3(c^3-27d)$ and $c_4=c(c^3-24d)$.
    As before, good reduction outside $S$ implies that for a $\mathfrak{q}\notin S$ we have that $v_{\fq}(\Delta_{\text{min}})=0$. So $\fq^{12k} || \Delta_E$ and $\fq^{4k} | c_4$ for some positive integer $k$. This yields to the following two relations
    $$\fq^{12k} || d^3(c^3-27d), \quad \fq^{4k}|c(c^3-24d).
    $$
    Now, we claim that $\fq^k | c$. Suppose not, by the second relation it follows that $\fq^{3k}|(c^3-24d)$. If $\fq^{3k} | d $, we get $\fq^{3k} | c^3$, which in turn gives $\fq^k | c$, a contradiction. So, $\fq^{3k} \nmid d $. The first relation then gives $\fq^{3k}|(c^3-27d)$. It follows that $\fq^{3k}|(c^3-24d)-(c^3-27d)=3d$. Since $\fq \notin S = \{\text{primes above }3 \}$, we get $\fq^{3k}| d$, another contradiction. As we exhausted all the possibilities, we can conclude that $\fq^k | c$. In particular, this gives $\fq^{3k}|c^3$. \\
    Secondly, we claim that $\fq^{3k}|d$. Suppose not, by the first relation we get $\fq^{3k}|(c^3-27d)$ and using $\fq^{3k}|c^3$ we get that $\fq^{3k}|d$, a contradiction. Hence $v_{\fq}(d):=t\geq 3k.$ In particular, so far we can deduce that $\fq^{3k}|(c^3-27d)$. By the first relation
    $$ 12k=v_{\fq}(d^3(c^3-27d))=3t+v_{\fq}(c^3-27d) \geq 3t + 3k
    $$
    This implies $t\leq 3k$, which gives $t=3k$ (as we have already shown $t\geq 3k$).
    Therefore, we can conclude $\mathfrak{q}^{k}|c$ and $\mathfrak{q}^{3k}||d$.
Hence, $$(c)\mathcal{O}_K=\displaystyle\prod_{\mathfrak{q}\notin S_K}\mathfrak{q}^{k_{\mathfrak{q}}+l_{\mathfrak{q}}} \prod_{\mathfrak{P}\in S_K}\mathfrak{P}^{c_\mathfrak{P}}, \:  (d)\mathcal{O}_K=\displaystyle\prod_{\mathfrak{q}\notin S_K}\mathfrak{q}^{3k_{\mathfrak{q}}} \prod_{\mathfrak{P}\in S_K}\mathfrak{P}^{d_\mathfrak{P}}$$ for some positive integers $c_{\mathfrak{P}}, d_{\mathfrak{P}},k_{\mathfrak{q}}, l_{\mathfrak{q}}$. Thus, as $\lambda=\frac{c^3}{d}$, we get
$$(\lambda)\mathcal{O}_K=I^3J \text{, where } I:=\displaystyle\prod_{\mathfrak{q}\notin S_K}\mathfrak{q}^{l_{\mathfrak{q}}}, \: J:=\prod_{\mathfrak{P}\in S_K}\mathfrak{P}^{3c_\mathfrak{P}-d_\mathfrak{P}}$$
which makes $J$ an $S$-ideal. 
\end{enumerate}
\end{proof}
\subsection{Modularity Results}\label{modsection}
We now carefully formulate modularity in the context of a totally real field. Let us first recall that given $K$ a totally real number field, $G_K$ its absolute Galois group and $E$ an elliptic curve over $K$, we say that $E$ is \textit{modular} if there exists a Hilbert cuspidal eigenform $\mathfrak{f}$ over $K$ of parallel weight $2$, with rational Hecke
eigenvalues, such that the Hasse–Weil L-function of $E$ is equal to the Hecke L-function
of $\mathfrak{f}$.  A more conceptual way to phrase this is that there is an isomorphism of
compatible systems of Galois representations
$$\rho_{E,p} \simeq \rho_{\mathfrak{f},p}$$
where the left-hand side is the Galois representation arising from the action of
$G_K$ on the $p$-adic Tate module $T_p(E)$, while the right-hand side is the Galois
representation associated to $\mathfrak{f}$. A comprehensive definition of \textit{Hilbert modular forms} and their associated representation can be found, for example in Wiles' \cite{AW}. In this paper we are mainly interested in the mod $p$ Galois representations and we denote their isomorphism by $\overline{\rho}_{E,p} \sim \overline{\rho}_{\mathfrak{f},p}$. 
\newpage
We need the following theorem proved by Freitas, Hung and Siksek in \cite{SHN}:
\begin{theorem} \label{modularity}
Let $K$ be a totally real field. There are at
most finitely many $\bar{K}$- isomorphism classes of non-modular elliptic curves $E$ over $K$. Moreover, if $K$ is real
quadratic, then all elliptic curves over $K$ are modular.
\end{theorem}
Furthermore Derickx, Najman and Siksek have recently proved in \cite{DNS}:
\begin{theorem} \label{modularity2}
Let $K$ be a totally real cubic number field and $E$ be an elliptic curve over $K$. Then $E$ is
modular.
\end{theorem}

\subsection{Irreductibility of\texorpdfstring{$\mod{p}$}{TEXT} representations of elliptic curves }\label{irreducibility}
We need the following theorem in the level lowering step of our proof. This was proved in \cite[Theorem 2]{SN1} and it is derived from the work of David and Momose who in turn built on Merel's Uniform Boundedness Theorem.
\begin{theorem} \label{irred}
Let $K$ be a Galois totally real field. There is an effective constant $C_K$, depending only on $K$, such that the following holds. If $p>C_K$ is prime, and $E$
is an elliptic curve over $K$ which has multiplicative reduction at all $\mathfrak{q}|p$, then $\overline{\rho}_{E,p}$ is irreducible.
\end{theorem} 
\begin{remark}
The above theorem is also true for any totally real field by replacing $K$ by its Galois closure.
\end{remark}
\subsection{Level lowering}\label{Hilbert}
We present a level lowering result proved by Freitas and Siksek in \cite{SN} derived from the work of Fujira \cite{19}, Jarvis \cite{25}, and Rajaei \cite{35}.
Let $K$ be a totally real field and $E/K$ be an elliptic curve of conductor $\mathcal{N}_E$. Let $p$ be a rational prime. Define
the following quantities:
\begin{equation}\label{condNM}
    \mathcal{M}_p=\prod_{\substack{\mathfrak{q}||\mathcal{N}_E \\ p|v_{\mathfrak{q}}(\Delta_\mathfrak{q})}}\mathfrak{q}, \text{ and } \mathcal{N}_p=\frac{\mathcal{N}_E}{\mathcal{M}_p}
\end{equation}
where $\Delta_{\mathfrak{q}}$ is the minimal discriminant of a local minimal model for $E$ at $\mathfrak{q}$. For a Hilbert eigenform $\mathfrak{f}$ over $K$, we write $\mathbb{Q}_{\mathfrak{f}}$ for the field generated by its eigenvalues.
\begin{theorem} \label{ll}
With the notation above, suppose the following statements hold:
\begin{enumerate}
    \item $p \geq 5$, the ramification index $e(\mathfrak{q}/p)<p-1$ for all $\mathfrak{q}|p$, and $\mathbb{Q}(\zeta_p)^+\nsubseteq K$,
    \item E is modular,
    \item $\overline{\rho}_{E,p}$ is irreducible,
    \item $E$ is semistable at all $\mathfrak{q}|p$,
    \item $p|v_{\mathfrak{q}}(\Delta_{\mathfrak{q}}) $ for all $\mathfrak{q}|p$.
\end{enumerate}
Then, there is a Hilbert eigenform $\mathfrak{f}$ of parallel weight $2$ that is new at level $\mathcal{N}_p$ and
some prime $\varpi$ of $\mathbb{Q}_{\mathfrak{f}}$ such that $\varpi|p$ and $\overline{\rho}_{E,p} \sim \overline{\rho}_{\mathfrak{f},\varpi}$.
\end{theorem}
\begin{proof}
A proof is given in \cite[p. 8]{SN}.
\end{proof}
\subsection{Eichler-Shimura}\label{ESconj}
For totally real fields, modularity reads as follows.
\begin{conjecture}[Eichler-Shimura]
Let $K$ be a totally real field. Let $\mathfrak{f}$ be a Hilbert
newform of level $\mathcal{N}$ and parallel weight $2$, with rational eigenvalues. Then there is
an elliptic curve $E_{\mathfrak{f}}/K$ with conductor $\mathcal{N}$ having the same L-function as $\mathfrak{f}$.
\end{conjecture}
Freitas and Siksek \cite{SN} obtained the following theorem  from works of Blasius \cite{DB}, Darmon \cite{D} and Zhang \cite{Z}.

\begin{theorem}\label{inertia}
 Let $E$ be an elliptic curve over a totally real field $K$ , and $p$ be an odd prime. Suppose that $\overline{\rho}_{E,p}$
 is irreducible, and $\overline{\rho}_{E,p} \sim \overline{\rho}_{\mathfrak{f},\varpi}$ for some Hilbert newform $\mathfrak{f}$ over $K$ of level $\mathcal{N}$ and parallel weight $2$ which
satisfies $\mathbb{Q}_{\mathfrak{f}} = \mathbb{Q}$. Let $\mathfrak{q}\nmid p$ be a prime ideal of $\mathcal{O}_K$ such that:
\begin{enumerate}
    \item $E$ has potentially multiplicative reduction at $\mathfrak{q}$,
    \item $p| \# \rep_{E,p}(I_{\mathfrak{q}})$,
    \item $p \nmid (\text{Norm}_{K/\mathbb{Q}}(\mathfrak{q})\pm 1) $.
\end{enumerate}
Then there is an elliptic curve $E_{\mathfrak{f}}/K$ of conductor $\mathcal{N}$ with the same L-function as $\mathfrak{f}$.
\end{theorem}

\section{Signature  \texorpdfstring{$(p,p,2)$}{TEXT}}\label{section3}
Let $K$ be a totally real field. Recall the set
$S_K=\{\mathfrak{P}\: : \mathfrak{P} \text{ is a prime of } K \text{ above } 2\}.$ Throughout this section we denote by $(a,b,c)\in \mathcal{O}_K^3$ a non-trivial, primitive solution of $a^p+b^p=c^2$.
\subsection{Frey Curve}
For $(a,b,c)\in \mathcal{O}_K^3$ as described above we associate the following Frey elliptic curve defined over $K$:
\begin{equation}\label{Frey2}
    E: Y^2=X^3+4cX^2+4a^pX.
\end{equation} 
 
We compute the arithmetic invariants: $$\Delta_E=2^{12}(a^2b)^p, c_4=2^6(4b^p+a^p)  \text{ and } j_E=2^6 \frac{(4b^p+a^p)
^3}{(a^2b)^p}.$$
\begin{lemma} \label{lemmane}
Let $(a, b, c)$ be the non-trivial, primitive solution to the equation $a^p+b^p=c^2$. Let $E$ be the associated Frey curve (\ref{Frey2}) with conductor $\mathcal{N}_E$. Then, for all primes $\fq\notin S_K$, the model $E$ is minimal, semistable and satisfies $p|v_{\mathfrak{q}}(\Delta_E)$. Moreover
\begin{equation}\label{ne}
\mathcal{N}_E=\displaystyle\prod_{\mathfrak{P}\in S_K}\mathfrak{P}^{r_{\mathfrak{P}}} \prod_{\substack{\mathfrak{q}|ab\\ \mathfrak{q}\notin S_K}}\mathfrak{q}, \qquad
\mathcal{N}_p=\displaystyle\prod_{\mathfrak{P}\in S_K}\mathfrak{P}^{r'_{\mathfrak{P}}}
\end{equation}
where $0\leq r'_{\mathfrak{P}} \leq r_{\mathfrak{P}}\leq 2+ 6v_{\mathfrak{P}}(2)$.
\end{lemma}
\begin{proof}
Let $\mathfrak{q}$ be an odd prime of $K$. The invariants of the model $E$ are $\Delta_E=2^{12}(a^2b)^p \text{ and } c_4=2^6(4b^p+a^p).$ Suppose that $\mathfrak{q}$ divides $\Delta_E$, so $\mathfrak{q}|ab$. Since $a$ and $b$ are relatively prime, $\mathfrak{q}$ divides exactly one of $a$ and $b$.
Therefore, $\mathfrak{q}$ does not divide $c_4$. In particular, the model is minimal at $\mathfrak{q}$ and has multiplicative reduction. Hence $p|v_{\mathfrak{q}}(\Delta_E)=v_{\mathfrak{q}}(\Delta_{\mathfrak{q}})$. On the other hand $\mathfrak{P} \in S_K$ is an even prime, so we have $r_{\mathfrak{P}}=v_{\mathfrak{P}}(\mathcal{N}_E)\leq 2 + 6v_{\mathfrak{P}}(2)$ by \cite[Theorem IV.10.4]{ASilv}.
The definition of $\mathcal{N}_E$ gives the desired form in (\ref{ne}). Then, use (\ref{condNM}) to get $\mathcal{N}_p$ and observe that $r'_{\mathfrak{P}}=r_{\mathfrak{P}}$ unless $E$ has multiplicative reduction at $\mathfrak{P}$ and $p|v_{\mathfrak{P}}(\Delta_{\mathfrak{P}})$ in which case $r_{\mathfrak{P}}=1$ and $r'_{\mathfrak{P}}=0$.
\end{proof}
\begin{lemma} \label{modularity3}
Let $K$ be a totally real field. There is some constant $A_K$ depending only on $K$, such that for any non-trivial, primitive solution $(a,b,c)$ of $a^p+b^p=c^2$ and $p>A_K$, the Frey curve given by (\ref{Frey2}) is modular.
\end{lemma}
\begin{proof}
By Theorem \ref{modularity}, there are at most finitely many possible $\bar{K}$-isomorphism classes of elliptic curves over $E$ which are not modular. Let $j_1, j_2, \dots, j_n \in K$ be the $j$-invariants of these classes. Define $\lambda := b^p/a^p$. The $j$-invariant of $E$ is
$$j(\lambda)=2^6(4\lambda +1)^3\lambda^{-1}
.$$
We can assume $\lambda \notin \{0, \pm 1 \}$ as these $\lambda$ lead to $j(\lambda) \in \mathbb{Q}$ and we know that all rational elliptic curves are modular.
Each equation $j(\lambda)=j_i$ has at most three solutions $\lambda \in K$. Thus there are values $\lambda_1,\dots,\lambda_m \in K\: (\text{where }m\leq 3n)$ such that if $\lambda \neq \lambda_k$ for all $k$, then the elliptic curve $E$ with $j$-invariant $j(\lambda)$ is modular.

If $\lambda = \lambda_k$ then $(b/a)^p = \lambda_k$, but the polynomial $x^p+\lambda_k$ has a root in $K$ if and only if $\lambda_k \in (K^*)^p$ because $K$ is totally real and $\lambda_k \notin \{ 0, \pm1 \}.$
Hence we get a lower bound on $p$ for each $k$, and by taking the maximum of these bounds we get $A_K$.

\begin{remark}
The constant $A_K$ is ineffective as the finiteness of Theorem \ref{modularity} relies on Falting's Theorem (which is ineffective). See \cite{SHN} for more details.
Note that if $K$ is quadratic or cubic we get $A_K=0$ (by the last part of Theorem \ref{modularity} and Theorem \ref{modularity2}).
\end{remark}

\end{proof}
\subsection{Images of Inertia}
We gather information about the images of inertia $\rep_{E,p}(I_{\mathfrak{q}})$. This is a crucial step in applying Theorem \ref{inertia} and for controlling the behaviour at the primes in $S_K$ of the newform obtained by level lowering.
\begin{lemma} \label{lemma}
Let $E$ be an elliptic curve over $K$ with $j$-invariant $j_E$. Let $p\geq 5$ and
let $\mathfrak{q} \nmid p$ be a prime of $K$. Then $p | \# \rep_{E,p}(I_{\mathfrak{q}})$ if and only if $E$ has potentially
multiplicative reduction at $\mathfrak{q}$ (i.e. $v_{\mathfrak{q}}(j_E)<0$) and $p \nmid v_{\mathfrak{q}}(j_E)$.
\end{lemma}
\begin{proof}
See \cite[Lemma 3.4]{SN}.
\end{proof}
\begin{lemma} \label{imgi}
Let ${\mathfrak{P}} \in S_K$ and $(a,b,c)$ a non-trivial, primitive solution to $a^p+b^p=c^2$ with $\mathfrak{P} | b$ and prime exponent $p>6 v_{\mathfrak{P}}(2)$. Let $E$ be the Frey curve in $(\ref{Frey2})$ with $j$-invariant $j_E$. Then $E$ has potentially multiplicative reduction at ${\mathfrak{P}}$ and $p | \# \rep_{E,p}(I_{\mathfrak{P}})$.
\end{lemma}
\begin{proof}
Assume that $\mathfrak{P} \in S_K$ with $v_{\mathfrak{P}}(b)=k$. Then $v_{\mathfrak{P}}(j_E)= 6v_{\mathfrak{P}}(2) - pk$. Since $p>6 v_{\mathfrak{P}}(2)$, it follows that $v_{\mathfrak{P}}(j_E)<0$ and clearly $p\nmid v_{\mathfrak{P}}(j_E)$.
This implies that $E$ has potentially multiplicative reduction at $\mathfrak{P}$ and by Lemma \ref{lemma} we get $p | \# \rep_{E,p}(I_{\mathfrak{P}})$.
\end{proof}
\subsection{Level Lowering and Eichler Shimura}
This is a key result in the proof of Theorem \ref{main1}, for which we have prepared the ingredients in the previous sections. We will follow the corresponding proofs in \cite{SN} and \cite{T}.
\begin{theorem}\label{LL}
Let $K$ be a totally real number field and assume it has a distinguished prime $\tilde{\mathfrak{P}} \in S_K$. Then there is a constant $B_K$
depending only on $K$ such that the following hold. Suppose $(a,b,c) \in \cO_K^3$ is a non-trivial, primitive solution to $a^p+b^p=c^2$ with prime exponent $p>B_K$ such that $\tilde{\fP}|b$. Write $E$
for the Frey curve (\ref{Frey2}). Then, there is an elliptic curve $E'$ over $K$ such that:
\begin{enumerate}
    \item the elliptic curve $E'$ has good reduction outside $S_K$; 
   \item $\overline{\rho}_{E,p}\sim\overline{\rho}_{E',p}$;
     \item $E'$ has a $K$-rational point of order $2$;
    \item $E'$ has potentially multiplicative reduction at $\tilde{\mathfrak{P}}$ $(v_{\tilde{\mathfrak{P}}}(j_{E'})<0$ where $j_{E'}$ is the $j$-invariant of $E')$.
\end{enumerate}
\end{theorem}
\begin{proof}
We first observe by Lemma \ref{lemmane} that $E$ has multiplicative reduction outside $S_K$. By taking $B_K$ sufficiently large, we see from Lemma \ref{modularity3} that $E$ is modular and by Theorem \ref{irred} that $\rep_{E,p}$ is irreducible. Applying Theorem \ref{ll} and Lemma \ref{lemmane} we see that $\rep_{E,p} \sim \rep_{\mathfrak{f},\varpi}$ for a Hilbert newform $\mathfrak{f}$ of level $\mathcal{N}_p$ and
some prime $\varpi|p$ of $\mathbb{Q}_{\mathfrak{f}}$. Here $\mathbb{Q}_{\mathfrak{f}}$ denotes the field generated by the Hecke eigenvalues $\mathfrak{f}$.
Next we reduce to the case when $\mathbb{Q}_{\mathfrak{f}}=\mathbb{Q}$, after possibly enlarging $B_K$. This step uses standard ideas originally due to Mazur that can be found in \cite[Section 4]{BS}, \cite[Proposition 15.4.2]{HC2}, and so we omit the details.
\par
Next we want to show that there is some elliptic curve $E'/K$ of conductor $\mathcal{N}_p$ having the same L-function as $\mathfrak{f}$.We apply
Lemma \ref{imgi} with $\mathfrak{P} = \tilde{\mathfrak{P}}$ and get that $E$ has potentially
multiplicative reduction at $\tilde{\mathfrak{P}}$ and $p | \# \rep_{E,p}(I_{\mathfrak{P}})$. The existence of $E'$ follows from Theorem \ref{inertia} after possibly enlarging $B_K$ to ensure that $p \nmid (\text{Norm}_{K/\mathbb{Q}}(\tilde{\mathfrak{\mathfrak{P}}})\pm 1)$.
By putting all the pieces together we can conclude that there is an elliptic curve $E'/K$ of conductor $\cN_p$ satisfying $\rep_{E,p} \sim \rep_{E',p}$. This proves (i) and (ii).

To prove (iii) we use that $\rep_{E,p} \sim \rep_{E',p}$ for some $E'/K$ with conductor  $\mathcal{N}_{p}$. After enlarging $B_K$ by an effective amount, and possibly replacing $E'$ by an isogenous curve, we may assume that $E'$ has a $K$-rational point of order $2$. This uses standard ideas which can be found, for example, in \cite[Section IV-6]{Serre}.

Now let $j_{E'}$ be the $j$-invariant of $E'$. As we have already seen, Lemma \ref{imgi} implies $p | \# \rep_{E,p}(I_{\tilde{\mathfrak{P}}})$, hence $p | \# \rep_{E',p}(I_{\tilde{\mathfrak{P}}})$, thus by Lemma \ref{lemma} we get that $E'$ has potentially multiplicative reduction at $\tilde{\mathfrak{P}}$ and so $v_{\tilde{\mathfrak{P}}}(j_{E'})<0$.
\end{proof}

\subsection{Proof of Theorem \ref{main1}}

\begin{proof}
Given a primitive, non-trivial solution $(a,b,c)$ such that $\tilde{\mathfrak{P}}|b$ with a prime exponent $p$ we associate the Frey elliptic curve in (\ref{Frey2}).
By Theorem \ref{LL} for there exists $B_K$ such that for all $p>B_K$ we can find an elliptic curve $E'$ which is related to $E$ by $\overline{\rho}_{E,p}\sim\overline{\rho}_{E',p}$ and has a $K$-rational point of order $2$. Hence by Lemma \ref{torsion} (i) we get a model
$$ E': Y^2=X^3+a'X^2+b'X
$$
with arithmetic invariants $\Delta_{E'}=2^{4}b'^2(a'^2-4b')$, $j_{E'}=2^8 \frac{(a'^2-3b')^3}{b'^2(a'^2-4b')}$.
Moreover, by Theorem \ref{LL} (i), we know that $E'$ has good reduction outside $S_K$ which implies that $v_{\mathfrak{q}}(j_{E'})\geq 0$ for $\mathfrak{q}\notin S_K$. Therefore, $j_{E'} \in  \mathcal{O}_{S_K}$.
Consider $\lambda:= \frac{a'^2}{b'}$ and $\mu:=\lambda-4=\frac{a'^2-4b'}{b'}$. 
Next, we need to show that $\lambda$ can be written as $\lambda = u\gamma^2$, where $u$ is an $S_K$-unit.
By Lemma \ref{mainelliptic} (i) applied to $E'$ we get that $$(\lambda) \cO_K = I^2J \text{ where } J \text{ is an } S_K \text{-ideal}. $$
Thus $[I]^2=[J]$ as elements of the class group Cl($K$) and $[J]\in \langle[{\mathfrak{P}}]\rangle_{\mathfrak{P}\in S_K}$. This implies that $[I] \in Cl_{S_K}(K)[2]$ and by our assumption on $K$ that $Cl_{S_K}(K)[2]$ is trivial, we get that $[I] \in \langle[{\mathfrak{P}}]\rangle_{\mathfrak{P}\in S_K}$, i.e. $I:=\gamma \Tilde{I}$, where $\Tilde{I}$ is an $S_K$-ideal and $\gamma \in \cO_K$.
Consequently, $$(\lambda)\mathcal{O}_K=(\gamma)^2 \Tilde{I}^2J \text{ where both }\Tilde{I} \text{ and } J \text{ are }S_K \text{-ideals.}$$
Finally, $(\frac{\lambda}{\gamma^2})\mathcal{O}_K$ is an $S_K$-ideal, which implies that $u:=\frac{\lambda}{\gamma^2}$ is an $S_K$-unit.
Now, by dividing $\mu + 4 = \lambda \text{ by }u$, we get
\begin{equation}\label{ab} 
    \alpha+\beta=\gamma^2, \qquad \alpha:=\frac{\mu}{u} \in \mathcal{O}_{S_K}^*, \qquad \beta:= \frac{4}{u} \in \mathcal{O}_{S_K}^*.
\end{equation}
Now, suppose that there is some $\tilde{\mathfrak{P}} \in S_K$ that satisfies $|v_{\tilde{\mathfrak{P}}}(\frac{\alpha}{\beta})|\leq 6v_{\tilde{\mathfrak{P}}}(2)$. We will show that $v_{\tilde{\mathfrak{P}}}(j_{E'}) \geq 0$, contradicting Theorem \ref{LL} (iv) and hence we can conclude the proof.
By using \eqref{ab} we can rewrite the assumption $|v_{\tilde{\mathfrak{P}}}(\frac{\alpha}{\beta})|\leq 6v_{\tilde{\mathfrak{P}}}(2)$ in terms of the valuation of $\mu$, using that $\frac{\alpha}{\beta}=\frac{\mu}{4}$: $$-4v_{\tilde{\mathfrak{P}}}(2)\leq v_{\tilde{\mathfrak{P}}}(\mu)\leq 8v_{\tilde{\mathfrak{P}}}(2).$$
Note that $j_{E'}=2^8(\mu+1)^3\mu^{-1}$, hence
 $$v_{\tilde{\mathfrak{P}}}(j_{E'})=8v_{\tilde{\mathfrak{P}}}(2)+3v_{\tilde{\mathfrak{P}}}(\mu+1)-v_{\tilde{\mathfrak{P}}}(\mu).$$ There are three cases according to the valuation of $\tilde{\mathfrak{P}}$ at $\mu$:
 \\
 \textbf{Case (1):} Suppose $v_{\tilde{\mathfrak{P}}}(\mu)=0$. This implies that $v_{\tilde{\mathfrak{P}}}(\mu+1)\geq 0$, thus $v_{\tilde{\mathfrak{P}}}(j_{E'})\geq0$, a contradiction.
\\
\textbf{Case (2):} Suppose $v_{\tilde{\mathfrak{P}}}(\mu)>0$. This implies $v_{\tilde{\mathfrak{P}}}(\mu +1)=0$, thus, by using $v_{\tilde{\mathfrak{P}}}(\mu)\leq 8v_{\tilde{\mathfrak{P}}}(2)$ we get again $v_{\tilde{\mathfrak{P}}}(j_{E'})\geq0$.
\\
\textbf{Case (3):} Finally, suppose $v_{\tilde{\mathfrak{P}}}(\mu)<0$. This implies $v_{\tilde{\mathfrak{P}}}(\mu +1)=v_{\tilde{\mathfrak{P}}}(\mu)$, thus, by using $-4v_{\tilde{\mathfrak{P}}}(2)\leq v_{\tilde{\mathfrak{P}}}(\mu)$, we get one last time $v_{\tilde{\mathfrak{P}}}(j_{E'})\geq0$.\\
All three cases lead to contradictions and hence we conclude the proof.
\end{proof}
\subsection{Proof of Theorem \ref{2inert}}
\begin{proof}
We want to apply Theorem \ref{main1} with $\tilde{\fP}=\fP$ and $S_K=\{\fP \}$. Note that $2 \nmid h_K^+$ implies that $Cl_{S_K}(K)[2]$ is trivial. As $2$ is inert, we get $v_{\fP}(2)=1$.

Now, let us consider the equation $\alpha + \beta = \gamma^2$, with $\alpha, \beta \in \mathcal{O}_{S_K}^*$. By scaling the equation by even powers of $2$ and swapping $\alpha$ and $\beta$ if necessary, we may assume $0 \leq v_{\fP}(\beta)\leq v_{\fP}(\alpha)$ with $v_{\fP}(\beta) \in \{0,1 \}$.
\\
\textbf{Case (1):} Suppose $v_{\fP}(\beta)=1$. If $v_{\fP}(\alpha)\geq 2$, then $v_{\fP}(\gamma^2)=v_{\fP}(\alpha + \beta)=1$, which leads to a contradiction as $v_{\fP}(\gamma^2)$ must be even. Thus, $v_{\fP}(\alpha)=v_{\fP}(\beta)=1$ and $v_{\fP}(\frac{\alpha}{\beta})=0\leq6$.
\\
\textbf{Case (2):} Suppose $v_{\fP}(\beta)=0$ with $\beta$ not a square. If $v_{\fP}(\alpha) > 6$, then $v_{\fP}(\gamma^2)=v_{\fP}(\alpha + \beta)=0$ and $\beta \equiv \gamma^2 \mod 2^6$. Consider the field extension $L=K(\sqrt{\beta})$. We will show that $L$ is unramified at $2$, hence contradicting $2 \nmid h_K^+$.
Consider the element $\delta:=\frac{\gamma+\sqrt{\beta}}{2}$. Its minimal polynomial is
$$ m_{\delta}(X) = X^2-\gamma X+ \frac{\gamma^2-\beta}{4}.
$$
This belongs to $\mathcal{O}_K[X]$ and has odd discriminant $\Delta= \beta$, proving that $L$ is unramified at $2$.
Thus, we must have $v_{\fP}(\alpha)\leq6$, giving $v_{\fP}(\frac{\alpha}{\beta})=v_{\fP}(\alpha)\leq6$.
\\
\textbf{Case (3):} Suppose $\beta$ is a square. By dividing everything through $\beta$, we may assume $\beta=1$. Then, by the hypothesis of the theorem we get \[v_{\fP}(\frac{\alpha}{\beta})=v_{\fP}(\alpha)\leq6.\]
\\
All of the possible three cases lead to $v_{\fP}(\frac{\alpha}{\beta})\leq 6=6v_{\fP}(2)$, so we can conclude the proof by Theorem \ref{main1}.
\end{proof}

\subsection{Proof of Theorem \ref{2quad}}
\begin{proof}\label{pf2q}
Note that the assumption $d \equiv 5 \mod 8$ gives that $2$ is inert in the quadratic field $K=\mathbb{Q}(\sqrt{d})$, take $\fP $ to be the unique prime above $2$ and denote $S_K= \{ \fP \}$.
Moreover, $d$ prime is equivalent to $2 \nmid h_K^+$ \cite[Section 1.3.1]{L}.
By Theorem \ref{2inert} it is enough to check that every solution $(\alpha, \gamma) \in \mathcal{O}^*_{S_K}\times \cO_{S_K}$ with $v_{\mathfrak{P}}(\alpha)\geq 0$ to the equation
$
\alpha + 1 = \gamma^2
$
satisfies $v_{\fP}(\alpha)\leq 6$.
Rearranging the above we get that $(\gamma+1)(\gamma-1)=\alpha$. Denote $x=\frac{(\gamma +1)}{2}$ and $y=\frac{(1-\gamma )}{2}$. 
Note that since $(\gamma+1), (\gamma-1) \in \mathcal{O}_{S_K}$ and they are factors of the $S_K$-unit $\alpha$, they must be $S_K$-units, consequently $x,y \in \mathcal{O}_{S_K}^*$.\par
In \cite[p. 15]{SN}, it is proved that the only solutions of $S_K$-unit equation $x+y=1$, where $K=\mathbb{Q}(\sqrt{d})$ with $d \equiv 5 \mod 8$, $d>5$ and $S_K= \{ \fP \}$ are the so-called \textit{irrelevant} solutions $(-1,2), (1/2,1/2), (2,-1)$. This leads to $\alpha \in \{-1,8\}$, and hence $v_{\fP}(\alpha) \in \{0,3\}$, proving $v_{\fP}(\alpha)\leq 6$. Thus we can conclude the proof by Theorem \ref{2inert}.
\end{proof}
\subsection{Proof of Theorem \ref{2local}}
\begin{proof}
We will take $\fP $ to be the unique prime above $2$ and denote $S_K= \{ \fP \}$. By Theorem \ref{2inert} it is enough to check that every solution $(\alpha, \gamma) \in \mathcal{O}^*_{S_K}\times \cO_{S_K}$ with $v_{\mathfrak{P}}(\alpha)\geq 0$ to the equation
$
\alpha + 1 = \gamma^2
$
satisfies $v_{\fP}(\alpha)\leq 6$.
Rearranging as in (\ref{pf2q}) we get an $S_K$-unit equation $x+y=1$ such that $\alpha =- 4xy$. \par We will now use a result proved in \cite[p.5]{SKN}. saying that if $K$ satisfies the hypothesis of Theorem \ref{2local}, it follows that every solution $(x,y)$ of the $S_K$-unit equation satisfies $\max\{v_{\fP}(x),v_{\fP}(y)\}<2v_{\fP}(2)=2$. Thus, $$v_{\fP}(\alpha)=2v_{\fP}(2)+v_{\fP}(x)+v_{\fP}(y)<2+2+2=6.$$ Hence we can conclude the proof by Theorem \ref{2inert}.
\end{proof}

\section{Signature  \texorpdfstring{$(p,p,3)$}{TEXT}}\label{section4}

Let $K$ be a totally real field. Recall the set
$S_K=\{\mathfrak{P}\: : \mathfrak{P} \text{ is a prime of } K \text{ above } 3\}.$ Throughout this section we denote by $(a,b,c)\in \mathcal{O}_K^3$ a non-trivial, primitive solution of $a^p+b^p=c^3$.
\subsection{Frey Curve}

For $(a,b,c)\in \mathcal{O}_K^3$ as described above we associate the following Frey elliptic curve defined over $K$:
\begin{equation}\label{Frey3}
    E: Y^2+3cXY+a^pY=X^3.
\end{equation}
We compute the arithmetic invariants:
 $$\Delta_E= 3^3(a^3b)^p,\: c_4=3^2c(9b^p+a^p) \text{ and }j_E=3^3 \frac{c^3(9b^p+a^p)^3}{(a^3b)^p}.$$
\begin{lemma} \label{lemmane2}
Let $(a, b, c)$ be the non-trivial, primitive solution to the equation $a^p+b^p=c^3$. Let $E$ be the associated Frey curve (\ref{Frey3}) with conductor $\mathcal{N}_E$. Then, for all primes $\fq\notin S_K$, the model $E$ is minimal, semistable and satisfies $p|v_{\mathfrak{q}}(\Delta_E)$. Moreover
\begin{equation}\label{ne2}
\mathcal{N}_E=\displaystyle\prod_{\mathfrak{P}\in S_K}\mathfrak{P}^{r_{\mathfrak{P}}} \prod_{\substack{\mathfrak{q}|ab\\ \mathfrak{q}\notin S_K}}\mathfrak{q}, \qquad
\mathcal{N}_p=\displaystyle\prod_{\mathfrak{P}\in S_K}\mathfrak{P}^{r'_{\mathfrak{P}}}
\end{equation}
where $0\leq r'_{\mathfrak{P}} \leq r_{\mathfrak{P}}\leq 2+ 3v_{\mathfrak{P}}(3)$.
\end{lemma}
\begin{proof}
The proof follows exactly like the proof of Lemma \ref{lemmane}.
\end{proof}
\begin{lemma} \label{modularity4}
Let $K$ be a totally real field. There is some constant $A_K$ depending only on $K$, such that for any non-trivial, primitive solution $(a,b,c)$ of $a^p+b^p=c^3$ the Frey curve given by (\ref{Frey3}) is modular.
\end{lemma}
\begin{proof}
The proof follows exactly like the proof of Lemma \ref{modularity3}.
\end{proof}
\subsection{Images of Inertia}
We need the following result about images of inertia whose prove follows exactly like the proof of Lemma \ref{imgi}, hence it is omitted.
\begin{lemma} \label{imgi2}
Let ${\mathfrak{P}} \in S_K$ and $(a,b,c) $ with $\mathfrak{P}|b$ and prime exponent $p>3v_{\mathfrak{P}}(3)$. Let $E$ be the Frey curve in $(\ref{Frey3})$ with $j$-invariant $j_E$. Then $E$ has potentially multiplicative reduction at ${\mathfrak{P}}$ and $p | \# \rep_{E,p}(I_{\mathfrak{P}})$.
\end{lemma}
\subsection{Level Lowering and Eichler Shimura}
As in the previous section, the crucial level lowering theorem reads as follows:
\begin{theorem}\label{LL3}
Let $K$ be a totally real number field and assume it has a distinguished prime $\tilde{\mathfrak{P}} \in S_K$. Then there is a constant $B_K$
depending only on $K$ such that the following hold. Suppose $(a,b,c) \in \cO_K^3$ is a non-trivial, primitive solution to $a^p+b^p=c^3$ with prime exponent $p>B_K$ such that $\tilde{\fP}|b$. Write $E$
for the Frey curve (\ref{Frey3}). Then, there is an elliptic curve $E'$ over $K$ such that:
\begin{enumerate}
    \item the elliptic curve $E'$ has good reduction outside $S_K$,
   \item $\overline{\rho}_{E,p}\sim\overline{\rho}_{E',p}$,
     \item $E'$ has a $K$-rational point of order $3$,
    \item $E'$ has potentially multiplicative reduction at $\tilde{\mathfrak{P}}$ $(v_{\tilde{\mathfrak{P}}}(j_{E'})<0$ where $j_{E'}$ is the $j$-invariant of $E')$.
\end{enumerate}
\end{theorem}
\begin{proof}
The proof follows exactly like the proof of Theorem \ref{LL}.
\end{proof}

\subsection{Proof of Theorem \ref{main2}}
\begin{proof}
Given a primitive, non-trivial solution $(a,b,c)$ such that $\tilde{\mathfrak{P}}|b$ with a prime exponent $p$ we associate the Frey elliptic curve in (\ref{Frey3}).
By Theorem \ref{LL3} for $p>B_K$ we can find an elliptic curve $E'$ which is related to $E$ by $\overline{\rho}_{E,p}\sim\overline{\rho}_{E',p}$ and has a $K$-rational point of order $3$. Hence by Lemma \ref{torsion} (ii) we get a model
$$ E': Y^2+c'XY+d'Y=X^3
$$
with arithmetic invariants $\Delta_{E'}=d'^3(c'^3-27d')$ and $j_{E'}=\frac{c'^3(c'^3-24d')^3}{d'^3(c'^3-27d')}$.

Moreover, by Theorem \ref{LL} (i), we know that $E'$ has good reduction outside $S_K$ which implies that $v_{\mathfrak{q}}(j_{E'})\geq 0$ for $\mathfrak{q}\notin S_K$. Therefore, $j_{E'} \in  \mathcal{O}_{S_K}$.
Consider $\lambda:= \frac{c'^3}{d'}$ and $\mu:=\lambda-27=\frac{c'^3-27d'}{d'}$. 
Next, we need to show that $\lambda$ can be written as $\lambda = u\gamma^3$, where $u$ is an $S_K$-unit.
By Lemma \ref{mainelliptic} (ii) applied to $E'$ we get that $$(\lambda) \cO_K = I^3J \text{ where } J \text{ is an } S_K \text{-ideal}. $$
Thus $[I]^3=[J]$ as elements of the class group Cl($K$) and $[J]\in \langle[{\mathfrak{P}}]\rangle_{\mathfrak{P}\in S_K}$. This implies that $[I] \in Cl_{S_K}(K)[3]$ and by our assumption on $K$ that $Cl_{S_K}(K)[3]$ is trivial, we get that $[I] \in \langle[{\mathfrak{P}}]\rangle_{\mathfrak{P}\in S_K}$, i.e. $I:=\gamma \Tilde{I}$, where $\Tilde{I}$ is an $S_K$-ideal and $\gamma \in \cO_K$.
Consequently, $$(\lambda)\mathcal{O}_K=(\gamma)^3 \Tilde{I}^3J \text{ where both }\Tilde{I} \text{ and } J \text{ are }S_K \text{-ideals.}$$
Finally, $(\frac{\lambda}{\gamma^3})\mathcal{O}_K$ is an $S_K$-ideal, which implies that $u:=\frac{\lambda}{\gamma^3}$ is an $S_K$-unit.
Now, by dividing $\mu + 27 = \lambda \text{ by }u$, we get
\begin{equation}\label{ab1} 
    \alpha+\beta=\gamma^3, \qquad \alpha:=\frac{\mu}{u} \in \mathcal{O}_{S_K}^*, \qquad \beta:= \frac{27}{u} \in \mathcal{O}_{S_K}^*
\end{equation}
Now, suppose that there is some $\tilde{\mathfrak{P}} \in S_K$ that satisfies $|v_{\tilde{\mathfrak{P}}}(\frac{\alpha}{\beta})|\leq 3v_{\tilde{\mathfrak{P}}}(3)$. We will show that $v_{\tilde{\mathfrak{P}}}(j_{E'}) \geq 0$, contradicting Theorem \ref{LL3} (iv) and hence we can conclude the proof.
By using \eqref{ab1} we can rewrite the assumption $|v_{\tilde{\mathfrak{P}}}(\frac{\alpha}{\beta})|\leq 3v_{\tilde{\mathfrak{P}}}(3)$ in terms of the valuation of $\mu$, using that $\frac{\alpha}{\beta}=\frac{\mu}{27}$:

 $$0\leq v_{\tilde{\mathfrak{P}}}(\mu)\leq 6v_{\tilde{\mathfrak{P}}}(3).$$ 
Note that $j_{E'}=(\mu+27)(\mu+3)^3\mu^{-1}$, hence
 $$v_{\tilde{\mathfrak{P}}}(j_{E'})= v_{\mathfrak{P}}(\mu+27)+3v_{\mathfrak{P}}(\mu+3)-v_{\mathfrak{P}}(\mu).$$ There are three cases according to the valuation of $\tilde{\mathfrak{P}}$ at $\mu$:
 \\
\textbf{Case (1):} Suppose $0\leq v_{\tilde{\mathfrak{P}}}(\mu)\leq v_{\tilde{\mathfrak{P}}}(3)$. This implies that $v_{\tilde{\mathfrak{P}}}(\mu+27)=v_{\tilde{\mathfrak{P}}}(\mu)$ and $v_{\tilde{\mathfrak{P}}}(\mu+3)\geq v_{\tilde{\mathfrak{P}}}(\mu)$, thus $v_{\tilde{\mathfrak{P}}}(j_{E'})\geq0$.
\\
\textbf{Case (2):} Suppose $v_{\tilde{\mathfrak{P}}}(3)< v_{\tilde{\mathfrak{P}}}(\mu)\leq 3v_{\tilde{\mathfrak{P}}}(3)$. This implies that $v_{\tilde{\mathfrak{P}}}(\mu+27)\geq v_{\tilde{\mathfrak{P}}}(\mu)$ and $v_{\tilde{\mathfrak{P}}}(\mu+3)= v_{\tilde{\mathfrak{P}}}(3)$, thus we get again $v_{\tilde{\mathfrak{P}}}(j_{E'})\geq0$.
\\
\textbf{Case (3):} Suppose $3v_{\tilde{\mathfrak{P}}}(3)< v_{\tilde{\mathfrak{P}}}(\mu)\leq 6v_{\tilde{\mathfrak{P}}}(3)$. This implies that $v_{\tilde{\mathfrak{P}}}(\mu+27) = 3v_{\tilde{\mathfrak{P}}}(3)$ and $v_{\tilde{\mathfrak{P}}}(\mu+3)= v_{\tilde{\mathfrak{P}}}(3)$, thus we get one last time $v_{\tilde{\mathfrak{P}}}(j_{E'})\geq0$.\\
All three cases lead to contradictions and hence we conclude the proof.
\end{proof}
\subsection{Proof of Theorem \ref{3inert}}
\begin{proof}
We want to apply Theorem \ref{main2} with $\tilde{\fP}=\fP$ and $S_K=\{ \fP \}$. As $3$ is inert, we get $v_{\fP}(3)=1$.

Now, let us consider the equation $\alpha + \beta = \gamma^3$, with $\alpha, \beta \in \mathcal{O}_{S_K}^*$. By scaling the equation by triple powers of $3$ and swapping $\alpha$ and $\beta$ if necessary, we may assume $0 \leq v_{\fP}(\beta)\leq v_{\fP}(\alpha)$ with $v_{\fP}(\beta) \in \{0,1,2 \}$. Also, we can assume that $\beta$ is positive, otherwise we multiply everything by $-1$.
\\
\textbf{Case (1):} Suppose $v_{\fP}(\beta)=2$. If $v_{\fP}(\alpha)\geq 3$, then $v_{\fP}(\gamma^3)=v_{\fP}(\alpha + \beta)=2$, which leads to a contradiction as $v_{\fP}(\gamma^3)$ must be a multiple of $3$. Thus, $v_{\fP}(\alpha)=v_{\fP}(\beta)=2$ and $v_{\fP}(\frac{\alpha}{\beta})=0<3$.
\\
\textbf{Case (2):} Suppose $v_{\fP}(\beta)=1$. If $v_{\fP}(\alpha)\geq 2$, then $v_{\fP}(\gamma^3)=v_{\fP}(\alpha + \beta)=1$, which leads to a contradiction as $v_{\fP}(\gamma^3)$ must be a multiple of $3$. Thus, $v_{\fP}(\alpha)=v_{\fP}(\beta)=1$ and $v_{\fP}(\frac{\alpha}{\beta})=0<3$.
\\
\textbf{Case (3):} Suppose $v_{\fP}(\beta)=0$ with $\beta $ not a cube. If $v_{\fP}(\alpha) > 3$, then $v_{\fP}(\gamma^3)=0$ and $\beta \equiv \gamma^3 \mod 3^4$. Consider the field extension $L=K(\sqrt[3]{\beta},\omega)$ of $K(\omega)$. We will show that $L$ is unramified at $3$, hence contradicting $3 \nmid h_{K(\omega)}$.

Consider the element $\delta:=\frac{\gamma^2+\gamma\omega\sqrt[3]{\beta}+\omega^2 \sqrt[3]{\beta}}{3}$. Its minimal polynomial is
$$ m_{\delta}(X) = X^3+ \gamma \frac{\gamma^3-\beta}{3}X^2-\gamma^2X + \frac{(\gamma^3-\beta)^2}{27}.
$$
This belongs to $\mathcal{O}_K[X]$ and has discriminant 
$$\Delta = -2\gamma^3\frac{(\gamma^3-\beta)^3}{3^5}-4\gamma^3\frac{(\gamma^3-\beta)^5}{3^9}+\gamma^6\frac{(\gamma^3-\beta)^2}{3^2}-4\gamma^6-\frac{(\gamma^3-\beta)^4}{3^3}.$$
We can deduce that $\Delta \equiv -4\gamma^6 \mod 3$, proving that $L$ is unramified at $3$.
Thus, we must have $v_{\fP}(\alpha)\leq 3$, giving $v_{\fP}(\frac{\alpha}{\beta})=v_{\fP}(\alpha)\leq 3$.
\\
\textbf{Case (4):} Suppose $\beta$ is a cube. By dividing everything through $\beta$, we can assume that $\beta = 1$.
Then by the hypothesis of the theorem, we get $v_{\fP}(\frac{\alpha}{\beta})=v_{\fP}(\alpha)\leq 3$.\\
All of the possible four cases lead to $v_{\fP}(\frac{\alpha}{\beta})\leq 3=3v_{\fP}(3)$, so we can conclude the proof by Theorem \ref{main2}.
\end{proof}
\subsection{Proof of Theorem \ref{3quad}}

\begin{proof}\label{pf3q}
Note that $d \equiv 2 \mod 3$ gives that $3$ is inert in the quadratic field $K=\mathbb{Q}(\sqrt{d})$, take $\fP $ to be the unique prime above $3$ and denote $S_K= \{ \fP \}$. By Theorem \ref{3inert} it is enough to check that every solution $(\alpha, \gamma) \in \mathcal{O}^*_{S_K}\times \cO_{S_K}$ with $v_{\mathfrak{P}}(\alpha)\geq 0$ to the equation
$
\alpha + 1 = \gamma^3
$
satisfies $v_{\fP}(\alpha)\leq 3$.\par 
Assume by a contradiction that we have a solution $\alpha$ to the above equation such that $v_{\fP}(\alpha)>3$. This implies that $v_{\fP}(\gamma)=0$, giving $\gamma \in \cO_K$.\\
Rearranging we get that $(\gamma-1)(\gamma-\omega)(\gamma-\omega^2)=\alpha$ when viewed over $L:=K(\omega)$. In the new field extension $L$ we have that $(3)\mathcal{O}_L=(\omega-1)^2\mathcal{O}_L$. We take $\fp=(\omega-1)\mathcal{O}_L$ and $S_L=\{ \fp \}$.
Denote $x:=\gamma -1,\: y:=\gamma-\omega,\: z:=\gamma-\omega^2$ and observe that
\begin{equation}\label{3eq}
    \begin{cases}

        x-y=(\omega-1) \\
        y-z=\omega(\omega-1)\\
       
\end{cases}
\end{equation}
Note that $x, y, z \in \mathcal{O}_{S_L}$ and they are factors of the $S_K$-unit $\alpha$, hence they must be $S_L$-units.\par 
Consider $\tau \in\text{Gal}(L/K)$ such that $\tau(\omega)=\omega^2$. It is easy to see that 
$$ \tau(x)=x ,\:\:\: \tau(y) = z \text{ and } \tau(\fp)=\fp.
$$
This implies that $v_{\fp}(y)=v_{\fp}(z)=:r$. We will show that $r=1$. Firstly note that by (\ref{3eq}) we get that  $1=v_{\fp}(\omega(\omega-1))=v_{\fp}(y-z)\geq r$. Suppose $r\leq 0$. Then $v_{\fp}(x)\geq v_{\fp}(xyz)=v_{\fp}(\alpha)\geq 8$ since $3^4|\alpha$.
Then, by using (\ref{3eq}) again, we will get $1=v_{\fp}(\omega-1)=v_{\fp}(x-y)=r\leq 0$, a contradiction. So, $r$ must be exactly $1$. As
$v_{\fp}(\alpha)=v_{\fp}(xyz)=8$, we must have $v_{\fp}(x)=6$.
Consider now 
$$u:=\frac{x}{\omega-1} \text{ and } v=\frac{-y}{\omega-1}.$$ 
By the above discussion, we will get that $\fp^5|u$ and $v \in \mathcal{O}_L^*$. Denote $F:= \mathbb{Q}(\omega)$.
As $v$ is a unit, we must have
\begin{equation}\label{Norm}
    \N_{L/F}(v) \in \cO_{F}^* = \langle \omega + 1\rangle
\end{equation}
As $u+v=1$, we get that $v \equiv 1 \mod 3$. Let $\sigma$ be the generator of $\text{Gal}(L/F)$.
By noting that $3|\sigma(u)$, we get that $\sigma(v) \equiv 1 \mod 3$ and consequently $\N_{L/F}(v)=v\sigma(v)\equiv 1 \mod 3$. This and (\ref{Norm}) give $\N_{L/F}(v)=1$. 
Suppose that $v \in \cO_L^*\setminus \cO_K^*=\omega\cO_K^*$, then $ \omega|\N_{L/F}(v) $ contradicting  $\N_{L/F}(v)=1$. Thus $v \in \cO_K^*$ giving $u=1-v \in \cO_K$ which is a contradiction as $u$ is a ratio of a $K$-integer and $\omega-1 \notin K$.
\end{proof}
\subsection{Proof of Theorem \ref{3local}}
We first need to prove some preliminary lemmas. Throughout this section, $K$ denotes a totally real field of degree $n$, $L:=K(\omega)$ and $F:=\mathbb{Q}(\omega)$. Moreover, $K$ satisfies the conditions (i), (ii), (iii) and (iv) in the statement of Theorem \ref{3local}. More precisely let $q$ be the prime which totally ramifies in $K$. Note that $q\geq5$ so it is inert in $F$. Denote $\tilde{\fq}:= (q)\cO_F$ and take $\fq$ to be the unique prime above $q$ in $L$, so that $(q)\cO_L=\fq^n \cO_L$. Take $\fP $ to be the unique prime above $3$ in $K$ and denote $S_K= \{ \fP \}$. In $L$ we have that $(3)\mathcal{O}_L=(\omega-1)^2\mathcal{O}_L$. We take $\fp=(\omega-1)\mathcal{O}_L$ and $S_L=\{ \fp \}$.
\begin{lemma}\label{L1}
Let $\lambda \in \cO_L$, then there exists $\beta \in \mathbb{Z}[\omega]$ such that $\lambda \equiv b \mod \fq$ and 
\begin{equation}\label{norm}
    \textup{Norm}_{L/F}(\lambda)\equiv b^n \mod \tilde{\fq}.
\end{equation}

\end{lemma}
\begin{proof}
Note that $\cO_L/\fq\cO_L \cong \mathbb{F}_q(\omega) \cong \mathbb{Z}[\omega]/q\mathbb{Z}[\omega]$. Thus, there exists $b \in \mathbb{Z}[\omega]$ such that $\lambda \equiv b \mod \fq$. Let $\bar{L}$ be the normal closure of $L$. Take $\sigma \in \text{Gal}(\bar{L}/F)$. Note that $$(\sigma(\fq \cO_{\bar{L}}))^n = \sigma (q\cO_{\bar{L}})=q\cO_{\bar{L}}=(\fq \cO_{\bar{L}})^n.$$
Thus, by the unique factorisation of ideals we get $\sigma(\fq \cO_{\bar{L}})=\fq \cO_{\bar{L}}$. Moreover, by applying $\sigma$ to $\lambda \equiv b \mod \fq$
we get that $\sigma(\lambda)\equiv b \mod \fq \cO_{\bar{L}}$. Finally multiplying everything together
$$\N_{L/F}(\lambda)=\prod_{\sigma}\sigma(\lambda) \equiv b^n \mod \fq \cO_{\bar{L}}.$$
As $\lambda \in \cO_L$, it follows that  $\N_{L/F}(\lambda) \in \cO_F$. Also $b^n \in \cO_F$. Thus, $\N_{L/F}(\lambda)-b^n \in \cO_F \cap \fq \cO_{\bar{L}}=\tilde{q} \cO_F.$ Hence (\ref{norm}) holds.
\end{proof}

\begin{lemma}\label{L2}
Suppose $\lambda \in \cO_L^*$ and (ii) holds, i.e. $\gcd (n, q^2-1) = 1$. Then $(\lambda \mod \fq) \in \langle \omega+1 \rangle =  \{\pm 1, \pm (\omega+1), \pm \omega  \}$.
\end{lemma}
\begin{proof}
Let $b \in \mathbb{Z}[\omega]$ with $\lambda \equiv b \mod \fq$ as in Lemma \ref{L1}. This gives us $\N_{L/F}(\lambda)\equiv b^n \mod \tilde{\fq}$. However, as $\lambda$ is a unit, we must have $$\N_{L/F}(\lambda) \in \cO_F^*=\langle \omega+1 \rangle.$$ Putting these together we get that $b^n \equiv (\omega +1)^i \mod \tilde{\fq}$.
On the other hand, 
$b\in \cO_F$ and maps to a non-zero element of
$ \cO_F /\tilde{\fq}\cO_F\cong \mathbb{F}_{q^2}$ thus $b^{q^2-1}\equiv  1 \mod \tilde{\fq}$.
The assumption $\gcd(n, q^2-1) = 1$ is equivalent to the existence of integers
$u, v $ so that $un + v(q^2-1) = 1$. It follows that 
$$ b = (b^n)^u(b^{q^2-1})^v \equiv (\omega +1 )^{iu} \mod \tilde{\fq}.
$$
Thus, $(\lambda \mod \fq) \in \langle \omega+1 \rangle =  \{\pm 1, \pm (\omega+1), \pm \omega  \}$.
\end{proof}
\begin{proof}[Proof of Theorem \ref{3local}]
 
We will reduce the problem to a simpler one as described in Section \ref{pf3q}. More precisely, by using Theorem \ref{3inert} and then rewriting the equation into an $S_K$-unit equation, we get that it is enough to show that there are no solutions to \begin{equation}\label{uv}
u+v=1
\end{equation}
with $(u,v)\in \cO_{S_L}^* \times \cO_L^*$ such that $\fp^5|u$. We will prove the slightly stronger statement that there are no solutions to (\ref{uv}) such that $9|u$.\\
Note that by (\ref{uv}) it follows that $v \equiv 1 \mod 9$. Thus $\sigma(v) \equiv 1 \mod 9$ for all conjugates $\sigma(v)$ of $v$ in $\text{Gal}(\bar{L}/F)$, where $\bar{L}$ is the normal closure of $L$. Hence, $\N_{L/F}(v)\equiv 1 \mod 9$. As $v$ is a unit, we get
$\N_{L/F}(v) \in \cO_F^*=\langle \omega+1 \rangle.$ Thus, the only possibility is 
\begin{equation}\label{normm}
    \N_{L/F}(v)=1.
\end{equation}
By Lemma \ref{L2} applied with $\lambda=v$ we get that 
\begin{equation}
    (v\mod \fq) \in \langle \omega+1 \rangle =  \{\pm 1, \pm (\omega+1), \pm \omega  \}.
\end{equation}
If $v \equiv 1 \mod \fq$, then $u = 1-v \equiv 0 \mod \fq$, so $\fq | u$, but this is false as $u$ is an $S_L$-unit and $\fp$ and $\fq$ are different primes.\\
Thus $(v\mod \fq) \in \{- 1, \pm (\omega+1), \pm \omega  \}$. Then 
\begin{equation}\label{N}
    (\N_{L/F}(v) \mod \fq) \in \{ (-1)^n, (\pm(\omega+1))^n, (\pm \omega)^n \}.
\end{equation} 
Since $\gcd (n,q^2-1)=1$ and $q \geq 5$ is a prime, it follows in particular that $2\nmid n$ and $3\nmid n$. This observation along with (\ref{N}) proves that $\N_{L/F}(v) \mod \fq \not\equiv 1$, contradicting (\ref{normm}).

\end{proof}

\section{$S$-unit equations and computability} \label{sunits}
Finally, we will describe how to algorithmically check the hypothesises in our two main Theorems \ref{main1} and \ref{main2} by studying how to compute solutions of certain (linear) \textit{$S$-unit equations} over the number field $K$, i.e. equations of the form $$ax+by=1 \text{ where } a,b \in K^* \text{ with solutions } x,y \in \mathcal{O}_{S}^*.$$
Throughout this section 
$S$ denotes a finite set of prime ideals of $K$.

\begin{theorem} [Siegel]\label{Sunits}
Let $K$ be a number field and $S \subset \mathcal{O}_K$ a finite set of prime ideals, and let $a,b \in K^*$. Then, the equation $$ax+by=1$$ has only finitely many solutions in $\mathcal{O}_S^*.$ 
\end{theorem}
\begin{remark} \label{sunit2}

Methods of effectively computing solutions to $S$-unit were pioneered by De Weger's famous thesis \cite{Weg} for $K=\mathbb{Q}$. His method of lattice approximation reduction algorithms was later generalized for all number fields by others, see for example Smart's \cite{Smart}. 
Moreover, an $S$-unit solver for $a=b=1$ has been implemented in the free open-source mathematics software, Sage by A. Alvarado, A. Koutsianas, B. Malmskog, C. Rasmussen, D. Roe, C. Vincent, M. West in \cite{AKMRVW}. 
\end{remark}
We will now study two non-linear equations involving $S$-units which are going to play a crucial role in checking the hypothesis of our Theorems \ref{main1} and \ref{main2}. 
Let $K$ be a number field and $S$ a finite set of prime ideals. Consider the equation
$$ \alpha + \beta = \gamma^i, \: \alpha, \beta \in \mathcal{O}_{S}^*,\:\: \gamma\in \mathcal{O}_{S}
.$$
There is a natural scaling action of $\cO_{S}^*$ on the solutions. We regard two solutions $(\alpha_1,\beta_1,\gamma_1) \sim_i (\alpha_2,\beta_2,\gamma_2)$ as equivalent if there is some $\epsilon \in \cO_{S}^*$ such that $\alpha_2=\epsilon^i \alpha_1$, $\beta_2=\epsilon^i \beta_1$ and $\gamma_2=\epsilon \gamma_1$. 

\begin{theorem}\label{equation}
Let $K$ be a number field and $S$ a finite set of prime ideals. Consider the equation
$$ \alpha + \beta = \gamma^i, \: \alpha, \beta \in \mathcal{O}_{S}^*,\:\: \gamma\in \mathcal{O}_{S}
.$$
For $i=2,3$, the equation has a finite number of solutions up to the equivalence relation $\sim_i$. Moreover, these are effectively computable.
\end{theorem}
\begin{proof}
Let $i=2$ and $(\alpha,\beta,\gamma) \in \cO^*_S \times \cO^*_S \times \cO_S$ a solution to $\alpha+\beta=\gamma^2$. By Dirichlet Unit Theorem $\cO^*_S$ is finitely generated, and hence $\cO^*_S/(\cO^*_S)^2$ is finite. Fix a set of representatives $\beta_1, \beta_2,..., \beta_l$. We may scale our solution so that $\beta \in \{ \beta_1, \beta_2, ..., \beta_l \}$. Thus, there are finitely many choices of $\beta$ (up to $\sim_2$ equivalence) and we fix one of them. We next show that for each such choice of $\beta$, there is a finite number of distinct $\alpha$, and thus, a finite number of triples $(\alpha, \beta,\gamma)$ up to $\sim_2$ equivalence.

We rewrite the equation as 
\begin{equation}\label{XYSunits}
    (\gamma+\sqrt{\beta})(\gamma-{\sqrt{\beta}})= \alpha \: \text{over } L
    .
    \end{equation}
where $L:=K(\sqrt{\beta})$. Denote by $ x:=\gamma+\sqrt{\beta}, y:=\gamma-\sqrt{\beta}
$ and consider $S':= \{ \fP_L \text{ prime of }L: \fP_L|\fP_K, \text{ for some } \fP_K \in S \}.$ 
We claim that $x,y$ are both $S'$-units in $L$. This follows by considering the valuation of the product in (\ref{XYSunits}) at the primes of $L$ outside the set $S'$. Then, we use the definition of $S'$ and the fact that $\alpha$ is an $S$-unit in $K$. Notice that
$$\frac{1}{2\sqrt{\beta}}x-\frac{1}{2\sqrt{\beta}}y=1.$$ 
By Theorem \ref{Sunits}, we get finitely many $S'$-unit solutions $x,y$, and thus finitely many possibilities for $\alpha = xy$. Moreover, these are computable by Remark \ref{sunit2}.
 
For $i=3$, the argument works in a similar manner. Fixing a representative $\beta$ of the finite quotient $\cO^*_S/(\cO^*_S)^3$, we rewrite the equation as
\begin{equation} \label{XY2}
    (\gamma-\sqrt[3]{\beta})(\gamma-\omega{\sqrt[3]{\beta}})(\gamma-\omega^2\sqrt[3]{\beta})= \alpha \: \text{over } L
    \end{equation}
where $L=K(\omega, \sqrt[3]{\beta})$ and $\beta \neq -1$. Denote by $x:= \gamma-\sqrt[3]{\beta}, y:=\gamma-\omega{\sqrt[3]{\beta}},$
$S':= \{ \fP_L \text{ prime of }L: \fP_L|\fP_K, \text{for some } \fP_K \in S \} .$ 

We make the quick note that for $\beta=-1$ we take $x:=\gamma+1$, $y=\gamma+\omega$, $L:=K(\omega)$ and the rest of the argument follows the same, so it is omitted.

As in the previous case, by examining the product in (\ref{XY2}) we get that $x,y$ are both $S'$-units in $L$ and 
$$ \frac{1}{(\omega-1)\sqrt[3]{\beta}}x-
\frac{1}{(\omega-1)\sqrt[3]{\beta}}y=1.
$$
Thus, by Theorem \ref{Sunits}, Remark \ref{sunit2} and the observation that $\alpha = xy(y-\omega(\omega-1)\sqrt[3]{\beta}) $, giving finely many numbers $\alpha$ for a fixed $\beta$ and so we conclude the proof.

\end{proof}
\begin{remark}
In the hypotheses of Theorems \ref{main1} and \ref{main2} one needs to examine the local behaviour of $\frac{\alpha}{\beta}$ which, by the above theorem, can only take a finite, computable number of values.
\end{remark}

\end{document}